\documentclass[reqno]{amsart}

\newif\ifanonymous
\anonymousfalse  

\usepackage{amssymb}
\usepackage{xcolor}
\usepackage[colorlinks,linkcolor=blue!50!black,citecolor=blue!50!black,urlcolor=blue!50!black]{hyperref}

\theoremstyle{plain}
\newtheorem{thm}{Theorem}[section]
\newtheorem{prop}[thm]{Proposition}
\newtheorem{cor}[thm]{Corollary}
\newtheorem{lem}[thm]{Lemma}

\newtheorem{introthm}{Theorem}

\theoremstyle{definition}
\newtheorem{defn}[thm]{Definition}

\theoremstyle{remark}
\newtheorem{rem}[thm]{Remark}

\newcommand{\m}{\mathfrak{m}}

\newcommand{\R}{\mathbb{R}}

\newcommand{\onabla}{\mathring \nabla}

\numberwithin{equation}{section}

\begin{document}
	
	\title[A conformal reduction for the X-ADM mass]{A conformal reduction for the X-ADM mass}
	
	\ifanonymous\else
	\author{Stephen McCormick}
	\address{Department of Engineering Sciences and Mathematics,
		Lule\aa{} University of Technology, 971 87 Lule\aa, Sweden}
	\email{stephen.mccormick@ltu.se}
	\thanks{The author is supported by the Olle Engkvist Foundation, the Magnus Bergvall Foundation, and foundations managed by the Royal Swedish Academy of Sciences}
	\fi
	

	\begin{abstract}
The X-ADM mass is a generalisation of the ADM mass for asymptotically flat manifolds, recently introduced by Mantegazza and Oronzio, who proved a positive mass theorem for this quantity in dimension $3$ under a certain topological condition. In this article, we prove the X-positive mass theorem with rigidity in all dimensions $n\geq3$ without any topological restrictions.

To this end, we first prove that the scalar curvature condition appearing in the X-positive mass theorem implies the metric is of positive Yamabe type. We then show that the X-ADM mass of the given metric is equal to the ADM mass of a scalar-flat representative in its conformal class plus a term that is explicitly non-negative. From this, the X-positive mass theorem follows from the standard positive mass theorem. As a corollary, we obtain the Riemannian mass--charge inequality of general relativity in all dimensions, which was previously only known for spin manifolds. Finally, we give a version of the X-positive mass theorem for a manifold with boundary.
\end{abstract}

\maketitle

\section{Introduction}
The Riemannian positive mass theorem for asymptotically flat manifolds is a cornerstone of mathematical general relativity and a central result in scalar curvature geometry. It has a long history with many contributors, but the first proofs were by Schoen and Yau \cite{SchoenYau1979PositiveMass} using minimal surface techniques, and by Witten \cite{Witten81} using a spinor argument. Witten's proof holds in all dimensions for spin manifolds \cite{Bartnik86} while Schoen and Yau's proof does not impose a spin condition and holds up to dimension $7$ \cite{schoen1979complete}. There has been considerable work by many authors extending the positive mass theorem to higher dimensions \cite{bi2026proof,chodosh2025generic,chodosh2023generic,Lohkamp2006,SchoenYauAllDimensions}, culminating in the recent work by Brendle and Wang \cite{BrendleWang2026} giving a full proof of the positive mass theorem in all dimensions.

Roughly, the positive mass theorem says that the ADM mass of an asymptotically flat manifold with non-negative scalar curvature is non-negative, and zero only when the manifold is isometric to $\R^n$. Recently, Mantegazza and Oronzio \cite{MantegazzaOronzio2026XADM} introduced a generalisation of the ADM mass depending on a given vector field $X$ called the X-ADM mass. If $X$ is a gradient vector field, then the X-ADM mass reduces to the weighted mass of Baldauf and Ozuch \cite{BaldaufOzuch22}, and if $X$ vanishes it reduces to the standard ADM mass. The positivity of the X-ADM mass also implies the mass--charge inequality of general relativity, essentially unifying several different positive mass-type theorems. Mantegazza and Oronzio were able to prove a positive mass theorem for this X-ADM mass (see Theorem \ref{thm-intro1} below) in dimension $3$, assuming an additional topological condition.

In this article, we show that the X-positive mass theorem is equivalent to the standard positive mass theorem by means of a conformal reduction argument. In particular, we show that positivity of the X-ADM mass follows from the standard positive mass theorem. The X-ADM mass of an asymptotically flat manifold equipped with a vector field $X$ is given in rectangular coordinates near infinity by
\begin{equation}
	\m_X(g)=\lim_{r\to \infty}\int_{S_r} \left(\partial_i g_{ij}-\partial_jg_{ii}+2X_j\right)\nu^j\,dS,
\end{equation}
where $S_r$ denotes a large sphere of radius $r$ in the exterior region and $\nu$ is the outward unit normal. We give precise definitions and conditions for this quantity to be well-defined and finite in the following section. A slightly simplified version of our main theorem is the following, wherein and throughout we denote the scalar curvature of a metric $g$ as $R(g)$. The reader is referred to Theorem \ref{thm-main1} for a precise statement.
\begin{introthm}[Theorem \ref{thm-main1}]\label{thm-intro1}
	Let $(M,g)$ be a smooth asymptotically flat $n$-manifold equipped with a vector field $X$ with suitable decay. 
	Suppose $R(g)+2\mathrm{div}(X)\in L^1$ and
	\begin{equation} \label{eq-introRcond}
		R(g)+2\mathrm{div}(X)\geq\frac{n-2}{n-1}|X|^2,
	\end{equation}
	then $\m_X(g)\geq0$ with equality if and only if $g=u^{4/(n-2)}g_{\mathbb R^n}$, $M\cong\mathbb R^n$, and $X=\nabla(\ln(u^{\frac{2(n-1)}{n-2}}))$, where $g_{\mathbb R^n}$ is (isometric to) the Euclidean metric.
\end{introthm}
\begin{rem}
	The proof of Theorem \ref{thm-intro1} in fact gives an expression for $\m_X$ from which positivity and rigidity follow directly. Namely we show
	\begin{equation}\label{eq-massdefect0}
		\m_X(g)=	\m(\varphi^{4/(n-2)}g)+\int_{M} c_n\left| \nabla\varphi+\frac{2}{c_n}X\varphi\right|^2+R_X(g)\varphi^2\,dV,	
	\end{equation}
	where $\varphi$ is a conformal factor that ensures $\varphi^{4/(n-2)}g$ is scalar-flat, $\m$ is the standard ADM mass, $c_n=\frac{4(n-1)}{n-2}$, and 
	\begin{equation*}
		R_X(g)=R(g)+2\mathrm{div}(X)-\frac{n-2}{n-1}|X|^2.
	\end{equation*}
	The proof relies on first proving that \eqref{eq-introRcond} implies that $g$ is of positive Yamabe type, which ensures the existence of the conformal factor $\varphi$.
\end{rem}

\begin{rem}
	Instead of condition \eqref{eq-introRcond}, Mantegazza and Oronzio considered the condition
	\begin{equation}\label{eq-MOcond}
		R(g)+2\mathrm{div}(X)\geq\left(1+\frac{1}{k}\right)|X|^2,
	\end{equation} 
	for $k\in(-\infty,-(n-1)]\cup(0,\infty)$, and specifically in the case $n=3$. However, this restriction on $k$ implies $(1+1/k)\geq\frac{n-2}{n-1}$, and therefore \eqref{eq-MOcond} implies \eqref{eq-introRcond}, so the X-positive mass theorem for all admissible values of $k$ follows from Theorem \ref{thm-intro1}. Furthermore, for any other choice of $k$, the constant on the right-hand side of \eqref{eq-introRcond} or \eqref{eq-MOcond} is larger, so the equality case immediately leads to $X\equiv0$ and rigidity of the classical positive mass theorem applies.
\end{rem}

By taking $X=\pm (n-1)E$ we see that Theorem \ref{thm-intro1} can be understood as an improved version of the well-known mass--charge inequality, which can be seen as a corollary of this theorem.
\begin{introthm}[Corollary \ref{cor-charge}]\label{thm-intromasscharge}
	Let $(M,g)$ be an asymptotically flat manifold equipped with a vector field $E$ with suitable decay, and $R(g), |\mathrm{div}(E)|\in L^1$ satisfy
	
	\begin{equation*}
		R(g)\geq 2(n-1)|\mathrm{div}(E)|+(n-1)(n-2)|E|^2.
	\end{equation*}
	
	Then
	
	\begin{equation*}
		\m(g)\geq |Q(E)|
	\end{equation*}
	
	where the ADM mass $\m(g)$ and total electric charge $Q(E)$ are given by
	
	\begin{equation}\begin{split}\label{eq-mE}
			\m(g)&=\lim_{r\to \infty}\int_{S_r} \left(\partial_i g_{ij}-\partial_jg_{ii}\right)\nu^j\,dS,\\ Q(E)&=\lim_{r\to\infty}\int_{S_r}2(n-1)E_i\nu^i\,dS.
	\end{split}\end{equation}
	
\end{introthm}

Note the normalisation we use for $\m$ and $Q$ differs from the physics literature by a dimensional constant. Theorem \ref{thm-intromasscharge} was first established in dimension $3$ by Gibbons and Hull \cite{GibbonsHull1982Bogomolny} and very recently Raulot established (the stronger initial data version of the) result in all dimensions $n\geq4$ under the assumption that $M$ is spin \cite{Raulot2025ChargedSpin} (see also, \cite{Raulot2026-2}). As a consequence of Theorem \ref{thm-intro1}, we immediately obtain Theorem \ref{thm-intromasscharge} in all dimensions without the spin assumption.

We also obtain a version of the main theorem in the case where $M$ has a non-empty boundary. The precise statement is given by Theorem \ref{thm-boundary}, but the simplified version is as follows.\\

\begin{introthm}[Theorem \ref{thm-boundary}]\label{thm-intro2}
	Let $(M,g)$ be an asymptotically flat $n$-manifold equipped with a vector field $X=o(r^{-n/2})$. 
	Suppose $R(g)+2\mathrm{div}(X)\in L^1$,
	
	\begin{equation*}
		R(g)+2\mathrm{div}(X)\geq\frac{n-2}{n-1}|X|^2,
	\end{equation*}
	
	and
	
	\begin{equation*}
		H\leq X\cdot\nu\,\text{ on }\partial M
	\end{equation*}
	
	where $H$ is the mean curvature of $\partial M$ with respect to the inward unit normal $\nu$, then $\m_X(g)>0$.
	
\end{introthm}

\begin{rem}
	Theorems \ref{thm-intro1} and \ref{thm-intro2} do not require that the scalar curvature be non-negative or integrable on its own, nor do they require that the ADM mass or the charge term are well-defined individually.
\end{rem}

\begin{rem}
	After the first version of this article appeared on the arXiv, a similar conformal reduction proof of Theorem \ref{thm-intro1} was reported by Oronzio \cite{oronzio2026weighted}. Instead of directly comparing to a scalar-flat metric, Oronzio decomposes the vector field $X$ into a gradient term plus a divergence-free term, so that the gradient term can be dealt with by the conformal factor of Law, Lopez and Santiago \cite{LawLopezSantiago2025WeightedMass}. Then the remaining divergence-free term is handled by a second conformal reduction obtaining a mass defect formula similar to \eqref{eq-massdefect0}.
\end{rem}

\section{Setup and definitions}
Throughout, we will always work on a smooth connected manifold $M$ of dimension $n\geq3$. Furthermore, we assume there exists a compact set $K\subset M$ and a diffeomorphism $\psi:M\setminus K\to \mathbb R^n\setminus \overline{B_1}$, where $\overline{B_1}$ denotes the closed unit ball. Fix some smooth background metric $\mathring g$ on $M$ agreeing with the pullback of the flat metric by $\psi$ on $M\setminus K$. It will be convenient to work in weighted Lebesgue and Sobolev spaces defined with respect to $\mathring g$, denoted by $L^{p}_\delta$ and $W^{k,p}_\delta$ respectively. We refer the reader to \cite{Bartnik86} for details on these spaces and weighted versions of various results for Sobolev spaces, but roughly $W^{k,p}_\delta$ can be thought of as functions with $W^{k,p}_{loc}$ regularity that behave like $o(r^\delta)$.

\begin{defn}\label{def-AF}
	We say $(M,g)$ is \textit{asymptotically flat} if $g-\mathring{g}\in W^{2,p}_\delta$ for some $p>\frac{n}{2}$ and $\delta\in(-(n-2),-\frac{n-2}{2}]$. \textit{We will assume $p$ and $\delta$ satisfy these restrictions throughout the entire document, for all weighted Sobolev and Lebesgue spaces used.}
\end{defn}

The ADM mass is a geometric quantity associated to an asymptotically flat manifold given by

\begin{equation}\label{eq-ADMdefn}
	\m(g)=\lim_{r\to\infty} \int_{S_r}\left( \onabla^ig_{ij}-\onabla_j(\mathring g^{ik}g_{ik})\right)\nu^j\,dS,
\end{equation}
where $\nu$ is the outward unit normal. Note that the definition often contains an additional multiplicative constant to agree with a physical definition of energy, which we omit here. A now well-known fact due to Bartnik \cite{Bartnik86} and Chru\'sciel \cite{Chrusciel1986} is that this quantity is a geometric invariant (independent of $\psi$). Furthermore, it is finite whenever the scalar curvature $R(g)$ is in $L^1$. The positive mass theorem states that if $R(g)\geq0$ then $\m(g)\geq0$ with equality if and only if $(M,g)$ is isometric to $(\mathbb R^n,\delta)$. Notice that we do not indicate which metric the surface element $dS$ is defined with respect to. This is because the asymptotics ensure it is independent of which asymptotically flat metric is used to define $dS$ and for this reason we omit it throughout.

Recently, Mantegazza and Oronzio introduced the X-ADM mass \cite{MantegazzaOronzio2026XADM}. Given a vector field $X$ on $M$, the X-ADM mass is given by
\begin{equation}
	\m_X(g)=\lim_{r\to \infty}\int_{S_r} \left(\onabla^i g_{ij}-\onabla_j(\mathring{g}^{ik}g_{ik})+2X_j\right)\nu^j\,dS,
\end{equation}
which is well-defined and finite if $R(g)+2\mathrm{div}(X)\in L^1$ (see Proposition \ref{prop-welldefmE} below). Note that whenever we write $\mathrm{div}$ it is to be understood to be with respect to $g$. In \cite{MantegazzaOronzio2026XADM}, the X-positive mass theorem (Theorem \ref{thm-intro1}) was proven in dimension $3$ under an additional topological assumption.

\begin{rem}
	Baldauf and Ozuch introduced a weighted mass quantity for manifolds with density \cite{BaldaufOzuch22} in the sense of Perelman's work \cite{perelman2002entropy}, and provided a spinor proof of its positivity.	Letting $X=\nabla f$, where $f$ is the function defining their measure $e^{-f}dV_g$, the X-ADM mass is exactly their weighted mass. In this sense, the X-ADM mass is a generalisation of the weighted mass.
	
	A similar conformal reduction idea to the one we develop here was used by Law, Lopez and Santiago to provide another proof of the positivity of the weighted mass \cite{LawLopezSantiago2025WeightedMass}. However, their method was to find a conformal factor that transformed the weighted mass into the ADM mass, while here we instead transform to a scalar flat metric and show that the X-ADM mass is bounded below by the ADM mass of the conformal metric. Their conformal reduction can also be used to prove a weighted Riemannian Penrose inequality and other weighted versions of known results \cite{MeWeighted}, but this relies on explicit knowledge of the conformal factor at the boundary, which we do not have here. So we cannot obtain a Penrose inequality for the X-ADM mass by this method.
\end{rem}

\section{Existence of the conformal metric}

We prove Theorem \ref{thm-intro1} by making a conformal transformation to a scalar-flat metric and applying the classical positive mass theorem there. In order to do this, we first must show that such a conformal transformation exists. The resolution of the Yamabe problem in this setting was first claimed by \cite{CantorBrill1981Laplacian} and corrected by Maxwell \cite{Maxwell2005ApparentHorizonBoundaries} (see also \cite{DiltsMaxwell2018YamabeAE}). This states that an asymptotically flat metric $g$ is conformal to a scalar flat asymptotically flat metric if and only if $g$ is of positive Yamabe type. That is to say, the Yamabe constant
\begin{equation}
	\mathcal Y([g])=\inf_{f\in C^\infty_c(M),f\not\equiv0}\frac{\int_M \frac{4(n-1)}{n-2}|\nabla f|^2+R(g)f^2\,dV_g}{\|f\|^2_{L^{2n/(n-2)}}},
\end{equation}
is positive. 

We show that the scalar curvature condition in the X-positive mass theorem suffices to ensure positive Yamabe type.
\begin{prop}\label{prop-Yamabe}
	Let $(M,g)$ be an asymptotically flat manifold in the sense of Definition \ref{def-AF} and let $X\in W^{2,p}_{\delta-1}$ be a vector field on $M$. Suppose
	\begin{equation}
		R(g)+ 2\mathrm{div}(X)\geq\frac{n-2}{n-1}|X|^2,
	\end{equation}
	then $g$ is positive Yamabe type.
\end{prop}

In order to establish this, we first need the following coercivity estimate.
\begin{prop}\label{prop-coerce}
	Fix some one-form $A\in L^\infty_{-1-\varepsilon}$ and consider the operator
	\begin{equation*}
		Lu= du+Au.
	\end{equation*}
	
	For $u\in W^{1,2}_{\delta}$, we have
	\begin{equation}\label{eq-coerce}
		\|u\|_{2^*,\delta}\leq C\|Lu\|_{2,\delta-1},
	\end{equation}
	where $2^*=\frac{2n}{n-2}$ and $C$ depends only on $g$ and $A$.
\end{prop}
We begin with the following lemma.
\begin{lem}\label{lem-ker}
	If $u\in W^{1,2}_\delta$ satisfies
	\begin{equation*}
		Lu=0
	\end{equation*}
	in the weak sense, then $u\equiv 0$.
\end{lem}
\begin{proof}
	In this proof and indeed throughout this entire section we will use $C$ to denote a constant that may change from line to line.
	
	We first work on an exterior region $E_R=M\setminus B_R$ for some large $R$. Recall the weighted Sobolev inequality \cite[Theorem 1.2]{Bartnik86}
	\begin{equation*}
		\|u\|_{2^*,\delta;E_R}\leq C\| u \|_{1,2,\delta;E_R}
	\end{equation*}
	and the weighted Poincar\'e inequality \cite[Lemma 3.10]{Bartnik2005PhaseSpace},
	\begin{equation*}
		\|u\|_{2,\delta;E_R}\leq C \|Du\|_{2,\delta-1;E_R}.
	\end{equation*}
	
	Combining these gives
	\begin{equation}
		\|u\|_{2^*,\delta;E_R}\leq C\| Du \|_{2,\delta-1;E_R}.
	\end{equation}
	From this we can prove that any $u\in W^{1,2}_\delta$ satisfying $Lu=0$ must vanish identically on $E_R$ for sufficiently large $R$. Assume $du+Au=0$ (in the weak sense) then we have
	\begin{equation}
		\| u \|_{2^*,\delta;E_R}\leq C\|Au\|_{2,\delta-1;E_R}\leq  C\|A\|_{n,-1;E_R}\| u\|_{2^*,\delta;E_R},
	\end{equation}
	where we also made use of the weighted H\"older inequality \cite[Theorem 1.2]{Bartnik86}. Now, since $C$ is independent of $R$ and $\|A\|_{n,-1;E_R}\to0$ as $R\to\infty$, we can choose $R$ large enough to ensure
	\begin{equation*}
		\| u \|_{2^*,\delta;E_R}\leq \frac12\| u\|_{2^*,\delta;E_R},
	\end{equation*}
	and therefore $u\equiv0$ on $E_R$. To prove that $u$ vanishes everywhere, we next show that there is a radius $\rho$ such that whenever $u$ vanishes on a ball of radius $\rho/2$ then it vanishes on a ball of radius $\rho$ centred at the same point (cf. \cite[Lemma 3.8]{Bartnik2005PhaseSpace}). Let $B_\rho$ denote the ball of radius $\rho$ centred at some fixed point with $u\equiv0$ on $B_{\rho/2}$, with $\rho$ taken sufficiently small that we can work in a local coordinate chart. For such a $u$, the standard Sobolev inequality holds (specifically we can apply Lemma 3.7 of \cite{Bartnik2005PhaseSpace}),
	\begin{equation*}
		\|u\|_{2^*;B_\rho}\leq c\|Du\|_{2;B_\rho}=c\|Au\|_{2;B_\rho}\leq c\|A\|_{n;B_\rho}\|u\|_{2^*;B_\rho},
	\end{equation*}
	where $c$ denotes a constant that can be chosen independent of $\rho$. Choosing $\rho=\rho_0$ small enough to ensure $c\|A\|_{n;B_{\rho_0}}<1$ ensures $u$ vanishes on $B_{\rho_0}$. Since $A\in L^n_{-1}(M)=L^n(M)$ and $u$ vanishes on $E_R$, $\rho_0$ can be chosen independent of the point where $B_\rho$ is centred. Covering $M\setminus E_{2R}$ by balls of radius $\rho_0$, we conclude that $u\equiv0$ on each small ball and therefore on all of $M$.
\end{proof} 
Note that in the preceding lemma, we only require $A\in L^n_{-1}$ rather than $L^\infty_{-1-\varepsilon}$. Furthermore, the argument holds even when a boundary is present, replacing balls with half balls intersecting the boundary where necessary.

We now use Lemma \ref{lem-ker} to prove Proposition \ref{prop-coerce}.
\begin{proof}[Proof of Proposition \ref{prop-coerce}]
	We prove \eqref{eq-coerce} by contradiction. To this end, we assume there exists a sequence $u_i\in W^{1,2}_\delta$ such that $\|u_i\|_{2^*,\delta}=1$ and $Lu_i\to0$ in $L^{2}_{\delta-1}$.
	
	We then have
	\begin{align}
		\| Du_i \|_{2,\delta-1}&\leq \|Lu_i\|_{2,\delta-1}+\|Au_i\|_{2,\delta-1}\nonumber \\
		&\leq \|Lu_i\|_{2,\delta-1}+\|A\|_{n,-1}\|u_i\|_{2^*,\delta}\\
		&\leq C,\nonumber
	\end{align}
	where we make use of the weighted H\"older inequality again. Combining this with the weighted Poincar\'e inequality on an exterior region $E_R=M\setminus \overline{B_R}$ then gives
	\begin{align*}
		\|u_i\|_{1,2,\delta}\leq C + \|u_i\|_{2;B_R}\leq C,
	\end{align*}
	
	where we also made use of $\|u_i\|_{2^*,\delta}=1$. Since $u_i$ is uniformly bounded in $W^{1,2}_\delta$, after passing to a subsequence if necessary, we have $u_i\rightharpoonup u\in W^{1,2}_\delta$ weakly. Furthermore, since $Lu_i\to 0$ in $L^{2}_{\delta-1}$ and $L$ is bounded and linear, we have $u\equiv0$. If we can show strong convergence of $u_i\to u$ in $L^{2^*}_\delta$ then we get a contradiction with $\|u_i\|_{2^*,\delta}=1$. Here is where we use $A\in L^\infty_{-1-\varepsilon}$, as we have
	\begin{align}\begin{split}
			\|u_i\|_{2^*,\delta}&\leq C \| u_i \|_{1,2,\delta}\\
			&\leq C\|Lu_i\|_{2,\delta-1}+C\|A\|_{\infty,-1-\varepsilon}\|u_i\|_{2,\delta+\varepsilon}+\|u_i\|_{2;B_R}\\
			&\leq C\left( \|Lu_i\|_{2,\delta-1}+\|u_i\|_{2,\delta+\varepsilon}\right).\label{eq-scalebroke}		\end{split}
	\end{align}
	By the weighted Rellich compactness theorem \cite[Lemma 2.1]{choquet1981elliptic} $W^{1,2}_\delta$ embeds compactly in $L^{2}_{\delta+\varepsilon}$ so in particular, \eqref{eq-scalebroke} shows $u_i\to 0$ in $L^{2^*}_\delta$ strongly, giving a contradiction.
	
\end{proof}

Now we are finally in a position to prove Proposition \ref{prop-Yamabe}.

\begin{proof}[Proof of Proposition \ref{prop-Yamabe}]
	By direct computation we have for any $f\in C^\infty_c$,
	\begin{align*}
		\int_M &\frac{4(n-1)}{n-2}|\nabla f|^2+R(g)f^2\,dV_g\geq \\
		&\geq \int_M \frac{4(n-1)}{n-2}|\nabla f|^2-\left( 2\mathrm{div}(X)-\frac{n-2}{n-1}|X|^2 \right)f^2\,dV\\
		&=\int_M \frac{4(n-1)}{n-2}|\nabla f|^2+4fX\cdot\nabla f+\frac{n-2}{n-1}|X|^2f^2\,dV\\
		&=\int_M \frac{4(n-1)}{n-2}\left|\nabla f +\frac{n-2}{2(n-1)}Xf   \right|^2\,dV.
	\end{align*}
	
	Now note that $W^{2,p}_{\delta-1}\subset L^\infty_{-1-\varepsilon}$ so Proposition \ref{prop-coerce} applies to the operator $L=d+\frac{n-2}{2(n-1)}X^\flat$. In particular, we apply Proposition \ref{prop-coerce} with $\delta=-(n-2)/2$ and then since the weighted space $L^{2^*}_{-(n-2)/2}$ is equal to the unweighted space $L^{2^*}$, we have
	\begin{equation*}
		\|f\|^2_{2^*}\leq C \int_M \frac{4(n-1)}{n-2}\left|\nabla f +\frac{n-2}{2(n-1)}Xf   \right|^2\,dV.
	\end{equation*}
	From this, we immediately see $\mathcal Y([g])>0$.
	
\end{proof}

\section{Proof of the main theorem}
We first quickly show that the X-ADM mass is well-defined under an appropriate integrability condition.
\begin{prop}\label{prop-welldefmE}
	For any asymptotically flat manifold $(M,g)$ and vector field $X\in W^{2,p}_{\delta-1}$ satisfying
	\begin{equation*}
		R(g)+2\mathrm{div}(X)\in L^1,
	\end{equation*}
	$\m_X(g)$ is well-defined and finite.
\end{prop} 
\begin{proof}
	It is well-known that
	\begin{equation*}
		\lim_{r\to\infty}\left( \int_{S_r}\left( \onabla^ig_{ij}-\onabla_j\mathring (g^{ik}g_{ik})\right)\nu^j\,dS-\int_{B_r}R(g)\,dV\right)
	\end{equation*}
	is well-defined and finite for all asymptotically flat $g$. Simply by adding and subtracting $2\mathrm{div}(X)$ to the bulk integral gives
	\begin{equation*}
		\lim_{r\to\infty}\left( \int_{S_r}\left( \onabla^ig_{ij}-\onabla_j(\mathring g^{ik}g_{ik})+2X_j\right)\nu^j\,dS-\int_{B_r}R(g)+2\mathrm{div}(X)\,dV\right).
	\end{equation*}
	This limit is still finite in general, so provided the bulk integral converges then so must the surface integral.
\end{proof}
\begin{rem}
	Geometric invariance of $\m_X$ follows by the same argument as geometric invariance of the ADM mass \cite{Bartnik86,Chrusciel1986}. However, we do not prove it directly as this is clearly seen to be inherited directly from the ADM mass via \eqref{eq-massdefect} below.
\end{rem}

Before proving the main theorem we give a brief lemma to justify discarding some terms at infinity in the calculation that follows.
\begin{lem}\label{lem-boundtrace}
	Given $v\in W^{2,p}_\delta$ and $u\in W^{1,p}_{\delta-1}$, we have 
	\begin{equation*}
		\lim_{r\to \infty} \int_{S_r} |vu|\,dS=0.
	\end{equation*}
\end{lem}
\begin{proof}
	This is essentially Lemmas 4.3 and 4.4 of \cite{Bartnik2005PhaseSpace}, but repeated for the spaces we use here. Recall that ranges of $\delta\in(-(n-2),-(n-2)/2]$ and $p>\frac{n}{2}$ are fixed, so in particular $v$ is H\"older continuous and $o(r^\delta)$, and therefore 
	\begin{equation}\label{eq-uv}
		\int_{S_r} |vu|\,dS=o(r^{\delta})\int_{S_r} |u|\,dS.
	\end{equation}

	Writing $A_r=\{r<|x|<2r\}$ and the usual rescaling argument (cf. \cite[Theorem 1.2]{Bartnik86}) combined with the Sobolev trace theorem on $A_1$ gives
	\begin{equation*}
		\|u\|_{1;S_r}= r^{n-1}\| u_r\|_{1;S_1} \leq C r^{n-1}\|u_r\|_{1,p;A_1}\leq C r^{n-1}r^{\delta-1}\|u\|_{1,p,\delta-1;A_r},
	\end{equation*}
	where $u_r(x)=u(rx)$. Noting that $\|u\|_{1,p,\delta-1;A_r}=o(1)$, combining this with \eqref{eq-uv} gives
	\begin{equation*}
		\int_{S_r} |vu|\,dS=o(r^{n+2\delta-2})
	\end{equation*}
	and since $\delta\leq -(n-2)/2$ by assumption, we are done.\end{proof}
\begin{thm}\label{thm-main1}
	Let $(M,g)$ be a smooth asymptotically flat manifold and $X\in W^{2,p}_{\delta-1}$ a vector field on $M$ satisfying
	\begin{equation} 
		R(g)+ 2\mathrm{div}(X)\geq\frac{n-2}{n-1}|X|^2,
	\end{equation}
	and $R(g)+ 2\mathrm{div}(X)\in L^1$.
	
	Then $\m_X(g)\geq0$ with equality if and only if $g=u^{4/(n-2)}g_{\mathbb R^n}$, $M\cong\mathbb R^n$, and $X=\nabla(\ln(u^{\frac{2(n-1)}{n-2}}))$, where $g_{\mathbb R^n}$ is (isometric to) the Euclidean metric.
\end{thm}
\begin{proof}
	By Proposition \ref{prop-Yamabe} and the resolution of the Yamabe problem in this setting \cite{Maxwell2005ApparentHorizonBoundaries}, there exists an asymptotically flat metric $\widetilde g=\varphi^{4/(n-2)}g$ with $R(\widetilde g)\equiv 0$ and $\varphi-1\in W^{2,p}_\delta\cap C^\infty$. Scalar curvature under a conformal change satisfies
	\begin{equation}\label{eq-Rconfchange}
		0=R(\widetilde g)=\varphi^{-\frac{n+2}{n-2}}\left( -c_n\Delta_g\varphi+R(g)\varphi \right),
	\end{equation}
	where $c_n=\frac{4(n-1)}{n-2}$. The standard formula for the change of ADM mass under a conformal transformation gives
	\begin{equation}\label{eq-masschange}
		\m(\widetilde g)=\lim_{r\to \infty}\int_{S_r} \left(\onabla^i g_{ij}-\onabla_j(\mathring g^{ik}g_{ik})-c_n\nabla_j\varphi\right)\nu^j\,dS\geq0.
	\end{equation}
	Note that since $R(\widetilde g)\equiv0$, $\m(\widetilde g)$ is well-defined and finite even when $\m(g)$ is not. In what follows, it will be convenient to write
	\begin{equation*}
		R_X(g)=	R(g)+ 2\mathrm{div}(X)-\frac{n-2}{n-1}|X|^2.
	\end{equation*}
	The key to the argument is to compare $\partial\varphi$ to $X$, which we do by integrating \eqref{eq-Rconfchange} multiplied by $\varphi^{\frac{2n}{n-2}}$. Working for the moment on a large ball $B_r$, we have
	\begin{align*}
		0&=\int_{B_r}-\varphi c_n\Delta_g\varphi+R(g)\varphi^2\,dV_g,
	\end{align*}
	which we can integrate by parts to obtain
	\begin{align*}	
		c_n\int_{S_r}\varphi&\nabla_i\varphi \nu^i\,dS=\int_{B_r}c_n|\nabla\varphi|^2+R(g)\varphi^2\,dV\nonumber\\
		&= \int_{B_r}c_n|\nabla\varphi|^2-2\mathrm{div}(X)\varphi^2+ \frac{n-2}{n-1}|X|^2\varphi^2+R_X(g)\varphi^2\,dV\nonumber\\
		&= \int_{S_r}-2\varphi^2 X_i\nu^i\,dS+\int_{B_r}c_n\left| \nabla\varphi+\frac{2}{c_n}X\varphi\right|^2+R_X(g)\varphi^2\,dV.
	\end{align*}
	
	Taking the limit, we want to replace $\varphi$ with $1$ in the surface integrals. Since $(\varphi-1),(\varphi^2-1)\in W^{2,p}_\delta$ and $\nabla\varphi,X\in W^{1,p}_{\delta-1}$, this is precisely what Lemma \ref{lem-boundtrace} allows us to do.
	
	In particular, we have
	
	\begin{align}
		\m_X(g)=\,&\,\lim_{r\to \infty}\int_{S_r} \left(\onabla^i g_{ij}-\onabla_j(\mathring g^{ik}g_{ik})+2X_j\right)\nu^j\,dS\nonumber\\\nonumber
		=\,&\, 	\lim_{r\to \infty}\int_{S_r} \left(\onabla^i g_{ij}-\onabla_j(\mathring g^{ik}g_{ik})-c_n\nabla_j\varphi\right)\nu^j\,dS\\\nonumber
		&+\int_{B_r} c_n\left| \nabla\varphi+\frac{2}{c_n}X\varphi\right|^2+R_X(g)\varphi^2\,dV\\
		=\,&\,\m(\widetilde g)+\int_{M} c_n\left| \nabla\varphi+\frac{2}{c_n}X\varphi\right|^2+R_X(g)\varphi^2\,dV,\label{eq-massdefect}
	\end{align}
	
	where $R_X(g)\in L^1$ by assumption and the decay for $\nabla\varphi$ and $X$ are precisely what is required to ensure the remaining term is finite. We see immediately that $R_X(g)\geq0$ implies $\m_X(g)\geq0$. Furthermore, in the case of equality, $R_X(g)=0$, $\m(\widetilde g)=0$ and
	
	\begin{equation}
		X=-\frac{c_n}{2\varphi}\nabla\varphi=\nabla f,
	\end{equation}
	
	where $f=\ln(\varphi^{-\frac{c_n}{2}})$, and by the rigidity of the standard positive mass theorem $\widetilde g=g_{\mathbb R^n}$.
	
\end{proof}

\begin{rem}
	Theorem \ref{thm-main1} (as well as Corollary \ref{cor-charge} and Theorem \ref{thm-boundary}) assumes $g$ is smooth so that $\varphi$ inherits that smoothness and therefore $\widetilde g$ is also a smooth metric. We only however need to ensure enough smoothness to apply the positive mass theorem to $(M,\widetilde g)$, so the result only requires $g$ be regular enough to ensure the conformal metric meets the requirements of the positive mass theorem to be applied.
\end{rem}

By applying Theorem \ref{thm-main1} to $X=\pm (n-1)E$, where $E$ is the electric field in the context of Einstein--Maxwell initial data we obtain the following immediate corollary.
\begin{cor}[Theorem \ref{thm-intromasscharge}]\label{cor-charge}
	Let $(M,g)$ be a smooth asymptotically flat manifold and $E\in W^{2,p}_{\delta-1}$ a vector field on $M$ satisfying
	\begin{equation*}
		R(g)\geq 2(n-1)|\mathrm{div}(E)|+(n-1)(n-2)|E|^2,
	\end{equation*}
	and $R(g),|\mathrm{div}(E)|\in L^1$.
	
	Then
	\begin{equation*}
		\m(g)\geq |Q(E)|,
	\end{equation*}	
	where $\m$ and $Q$ are given by \eqref{eq-mE}.
\end{cor}

\begin{rem}
	Around the same time this article was completed, Raulot investigated the equality case for the mass--charge inequality \cite{Raulot2026-2} using a spinor method. He defines a notion of charged parallel spinors, which are spinors that are parallel with respect to a connection that is modified by the electric field $E$ and uses them with a charged Lichnerowicz formula to study the mass--charge inequality. Interestingly, he also proves that a manifold carries a charged parallel spinor if and only if it is conformal to a manifold that carries a parallel spinor \cite[Proposition 3]{Raulot2026-2}, with the same conformal factor that appears in the rigidity case for Theorem \ref{thm-main1}.
\end{rem}

\section{X-positive mass theorem for manifolds with boundary}
Throughout, we have assumed that $M$ has no boundary for simplicity and the sake of presentation. However, the Yamabe problem is well understood in the case that $M$ has a compact boundary so the argument carries through almost identically provided that the boundary satisfies a kind of flux condition for the X-ADM mass. We present it separately here in this section for the sake of exposition.

\begin{thm}\label{thm-boundary}
	Let $(M,g)$ be an asymptotically flat manifold with non-empty compact boundary $\partial M$, and $X\in W^{2,p}_{\delta-1}$ a vector field on $M$ satisfying
	\begin{equation} 
		R(g)+ 2\mathrm{div}(X)\geq\frac{n-2}{n-1}|X|^2,
	\end{equation}
	and $R(g)+ 2\mathrm{div}(X)\in L^1$.
	
	Assume additionally that
	\begin{equation}\label{eq-boundcond}
		H\leq X\cdot\nu\,\text{ on }\partial M
	\end{equation}
	where $H$ is the mean curvature of $\partial M$ with respect to the inward (pointing towards infinity) unit normal $\nu$,
	then $\m_X(g)>0$.
\end{thm}
\begin{proof}
	The Yamabe constant in this case is given by \cite{Maxwell2005ApparentHorizonBoundaries}
	\begin{equation}
		\mathcal Y([g])=\inf_{f\in C^\infty_c(\overline M),f\not\equiv0}\frac{\int_M \frac{4(n-1)}{n-2}|\nabla f|^2+R(g)f^2\,dV_g-2\int_{\partial M}Hf^2\,dS_g}{\|f\|^2_{L^{2n/(n-2)}}}.
	\end{equation}
	Essentially the same calculation as in the proof of Proposition \ref{prop-Yamabe} gives for any $f\in C^\infty_c$,
	\begin{align*}
		\int_M &\frac{4(n-1)}{n-2}|\nabla f|^2+R(g)f^2\,dV_g\geq \\ &\geq \int_M \frac{4(n-1)}{n-2}|\nabla f|^2-\left( 2\mathrm{div}(X)-\frac{n-2}{n-1}|X|^2 \right)f^2\,dV\\&=\int_M \frac{4(n-1)}{n-2}\left|\nabla f +\frac{n-2}{2(n-1)}Xf   \right|^2\,dV+\int_{\partial M} 2f^2X_i\nu^i\,dS_g,
	\end{align*}
	where the only difference is an additional boundary term on $\partial M$ arising from the divergence theorem. From Proposition \ref{prop-coerce} and \eqref{eq-boundcond} we get $\mathcal Y([g])>0$ again.
	
	By Proposition 4.1 of \cite{Maxwell2005ApparentHorizonBoundaries}, there exists $\varphi>0$ with $(\varphi-1)\in W^{2,p}_\delta$ such that $\widetilde g=\varphi^{4/(n-2)}g$ is scalar flat and $\partial M$ is a minimal surface with respect to $\widetilde g$.
	
	Following the proof of Theorem \ref{thm-main1}, we have
	\begin{align*}
		0=\,&\,\int_{B_r}-\varphi c_n\Delta_g\varphi+R(g)\varphi^2\,dV_g,
	\end{align*}
	which gives
	\begin{align*}
		c_n\int_{S_r}&\varphi\nabla_i\varphi \nu^i\,dS-c_n\int_{\partial M}\varphi\nabla_i\varphi \nu^i\,dS=\int_{B_r}c_n|\nabla\varphi|^2+R(g)\varphi^2\,dV\nonumber\\
		=\,&\, \int_{B_r}c_n|\nabla\varphi|^2-2\mathrm{div}(X)\varphi^2+ \frac{n-2}{n-1}|X|^2\varphi^2+R_X(g)\varphi^2\,dV\\
		=\,&\, 2\int_{\partial M}\varphi^2 X_i\nu^i\,dS_g-2\int_{S_r}\varphi^2 X_i\nu^i\,dS\\
		&+\int_{B_r}c_n\left| \nabla\varphi+\frac{2}{c_n}X\varphi\right|^2+R_X(g)\varphi^2\,dV,
	\end{align*}
	and then making use of the change of mean curvature formula for conformal transformations 
	\begin{equation}\label{eq-confH}
		0=H_{\widetilde g}=\varphi^{-\frac{n}{n-2}}\left( \frac{2(n-1)}{n-2}\nabla_i\varphi\nu^i+H\varphi \right),
	\end{equation}
	we have
	\begin{align*}
		c_n\int_{S_r}\varphi\nabla_i\varphi \nu^i\,dS=\,&\,-2\int_{S_r}\varphi^2 X_i\nu^i\,dS+2\int_{\partial M} (X_i\nu^i-H)\varphi^2\,dS_g\\
		& +\int_{B_r}c_n\left| \nabla\varphi+\frac{2}{c_n}X\varphi\right|^2+R_X(g)\varphi^2\,dV.
	\end{align*}
	Continuing as in the proof of Theorem \ref{thm-main1} we arrive at
	\begin{align*}
		\m_X(g)=\,&\,\m(\widetilde g)+\int_{M} c_n\left| \nabla\varphi+\frac{2}{c_n}X\varphi\right|^2+R_X(g)\varphi^2\,dV\\
		&+2\int_{\partial M}(X_i\nu^i-H)\varphi^2\,dS_g.
	\end{align*}
	We again obtain $\m_X(g)\geq \m(\widetilde g)$, and since $\partial M$ is a minimal surface boundary for $(M,\widetilde g)$ the positive mass theorem applies to $(M,\widetilde g)$. In fact, since $\partial M$ is non-empty, the rigidity case cannot hold and therefore we conclude strict positivity.
	
\end{proof}

\bibliographystyle{amsplain}
\bibliography{Refs}

\end{document}